\documentclass{article}
\usepackage{fullpage}
\usepackage{amsmath}
\usepackage{amsthm}
\usepackage{amssymb}
\usepackage{url}
\usepackage{graphicx}
\usepackage{verbatim}
\usepackage{url}
\usepackage{graphicx}
\usepackage{tikz}
\usetikzlibrary{calc}
\usetikzlibrary{patterns}
\tikzstyle{mybox} = [draw=black, fill=white,  thick,
    rectangle, inner sep=10pt, inner ysep=20pt]
\tikzstyle{mybox} = [draw=black, fill=white,  thick,
    rectangle, inner sep=2pt, inner ysep=2pt]

\newtheorem{thm}{Theorem}

\newtheorem{cor}{Corollary}
\newtheorem{prop}{Proposition}
\theoremstyle{definition}
\newtheorem{definition}{Definition}
\newtheorem{remark}{Remark}
\newtheorem{example}{Example}

\begin{document}

%\title{A Geometric Algorithms for Solving a Linear System}
%\title{New Iterative Methods for Solving Linear Systems Via The Triangle Algorithm}
%\title{Iterative Methods for Solving Linear Systems Via The Convex Hull Triangle Algorithm}
\title{Solving Linear System of Equations Via A Convex Hull Algorithm}
%\title{Working Paper}
\author{Bahman Kalantari \\
Department of Computer Science, Rutgers University, NJ\\
kalantari@cs.rutgers.edu
}
\date{}
\maketitle

\begin{abstract}
We present new iterative algorithms for solving  a square linear system $Ax=b$ in dimension $n$ by employing the {\it Triangle Algorithm} \cite{kal12}, a fully polynomial-time approximation scheme for testing if the convex hull of a finite set of points in a Euclidean space contains a given point.  By converting  $Ax=b$ into a convex hull problem and solving via the Triangle Algorithm, together with a {\it sensitivity theorem},  we compute in $O(n^2\epsilon^{-2})$ arithmetic operations an approximate solution satisfying $\Vert Ax_\epsilon - b \Vert \leq \epsilon \rho$, where $\rho= \max \{ \Vert a_1 \Vert, \dots, \Vert a_n \Vert, \Vert b \Vert \}$, and $a_i$ is the $i$-th column of $A$.  In another approach we apply the Triangle Algorithm incrementally, solving a sequence of convex hull problems while repeatedly employing a {\it distance duality}.  The simplicity and theoretical complexity bounds of the proposed algorithms, requiring no structural restrictions on the matrix $A$, suggest their potential practicality, offering alternatives to the existing exact and iterative methods, especially for large scale linear systems. The assessment of computational performance however is the subject of future
experimentations.
\end{abstract}

%\newpage

\section{Introduction}

The significance of methods for solving a linear system of equations becomes evident with high school mathematics and science classes.  Solving linear system of equations is undoubtedly one of the most practical problems in numerous aspects of scientific computing.  Gaussian elimination is the most familiar method for solving a linear system of equations, discussed in numerous books, e.g. Atkinson \cite{Atkinson}, Bini and Pan \cite{Bini}, Golub and van Loan \cite{Golub}, and Strang \cite{Strang}.  The method itself is an important motivation behind the study of linear algebra.   Iterative methods for solving linear systems offer very important alternatives to direct methods and find  applications in problems that require the solution of very large or sparse linear systems.  For example, problems from discretized partial differential equations lead to large sparse  systems of equations making the use of direct methods impractical.

Iterative methods generate a sequence of approximate solutions. They begin with an initial approximation and successively improve it until a desired approximation is reached satisfying a certain measure of error. Iterative methods include such classical methods as the Jacobi, the Gauss-Seidel, the successive over-relaxation (SOR), and the symmetric successive over-relaxation method which applies to the case when the coefficient matrix is symmetric. When the coefficient matrix is symmetric and positive definite the conjugate gradient method (CG) also becomes applicable. Convergence rate of iterative methods can often be substantially accelerated by preconditioning.  Some major references in the vast subject of iterative methods include, Barrett et al. \cite{Bar}, Golub and van Loan \cite{Golub},  Greenbaum \cite{Green},  Saad \cite{Saad}, van der Vorst \cite{van1}, Varga \cite{Var}, and Young \cite{Young}.

To guarantee convergence, iterative methods often require the coefficient matrix to satisfy certain conditions, or decompositions, necessary to carry out the iterations.  Some of the theoretical analysis, especially in the earlier works on iterative methods,  is only concerned with convergence or conditions on the coefficient matrix or its decompositions that guarantee convergence. However, in some cases theoretical complexity analysis is provided, see e.g. Reif \cite{Reif} who considers the complexity of iterative methods for sparse diagonally dominant matrices.

The major computational effort in each iteration of the iterative methods involves matrix-vector multiplication. This makes iterative methods very attractive for solving large systems and also for parallelization.  There is also a vast literature on parallelization of iterative methods for large systems, see
Demmel\cite{Dem}, Dongarra et al. \cite{Don},
Duff and van der Vorst \cite{Duff},  van der Vorst  \cite{van1}, van der Vorst and Chan \cite{Van2}.

In this article we offer a very different iterative method for solving an  $n \times n$ linear system $Ax = b$.  It is inspired by a new geometric algorithm for a convex hull problem in \cite{kal12}.  To illustrate the approach to be presented,  we begin with a very simple example. Consider a triangle with vertices $v_1 =(v_{11}, v_{12})$, $v_2 =(v_{21}, v_{22})$, $v_3 =(v_{31}, v_{32})$, and a point $p=(p_1,p_2)$ in their convex hull (Figure \ref{Fig1}).

\begin{figure}[htpb]
	\centering
	
	\begin{tikzpicture}[scale=0.5]
%\draw (0, 10) rectangle (10, 0);	
	
%%here

		\begin{scope}[black]
		\draw (0,0) -- (7,0) -- (4,3) -- cycle;
		\draw (0,0) node[below] {$v_1$};
		\draw (7,0) node[below] {$v_2$};
		\draw (4,3) node[above] {$v_3$};
		\end{scope}
		\filldraw (4.1,0.8) circle (2pt) node[left] {$p$};
\filldraw (0,0) circle (2pt);
\filldraw (7,0) circle (2pt);
\filldraw (4,3) circle (2pt);
%%tohere

	\end{tikzpicture}
	\caption{A point $p$ in the convex hull of vertices of a triangle.}
\label{Fig1}
\end{figure}
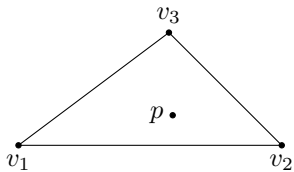

In particular, there exists nonnegative scalars $\alpha_1, \alpha_2, \alpha_3$ such that
\begin{equation}
\alpha_1 v_1+\alpha_2 v_2+\alpha_3 v_3=p, \quad  \alpha_1+\alpha_2+\alpha_3=1.
\end{equation}

When the vertices are not collinear, the scalars can be determined by solving a linear system having an invertible coefficient matrix:

\begin{equation} \label{matrixeq}
\begin{pmatrix} v_{11}& v_{21}&v_{31}\\
v_{12}& v_{22}&v_{32}\\
1& 1&1\\
\end{pmatrix}
\begin{pmatrix} \alpha_1\\
\alpha_2\\
\alpha_3\\
\end{pmatrix}=
\begin{pmatrix} p_1\\
p_2\\
1\\
\end{pmatrix},
\end{equation}

However, since $p$ is known to lie in the convex hull of the vertices, we can solve this system using a new geometric approach.  In this article we will first review our iterative algorithm for solving the general case of this convex hull problem.  Then, we will show how to convert the general linear system $Ax=b$ into a convex hull problem so that the convex hull algorithm would be applicable.

In general the convex hull of a set $S$ in a Euclidean space is the intersection of all convex sets that contain it.  In particular, if $S= \{v_1, \dots, v_n\} \subset \mathbb{R} ^m$, its convex hull
written as $conv(S)$ is given by
\begin{equation}
conv(S)= \bigg \{\sum_{i=1}^n \alpha_i v_i,  \quad \sum_{i=1}^n \alpha_i=1, \quad \alpha_i \geq 0, \quad \forall i=1, \dots, n \bigg \}
\end{equation}
In this article we describe new iterative methods for solving a linear system of equations utilizing a new algorithm, called the {\it Triangle Algorithm} \cite{kal12}, designed to solve the following {\it convex hull decision problem}:

Given a set of points $S= \{v_1, \dots, v_n\} \subset \mathbb{R} ^m$ and a distinguished point $p \in \mathbb{R} ^m$, test if $p$ lies in  $conv(S)$. A practical approximate version of the convex hull decision problem is the following:

Given $\epsilon \in (0,1)$, either compute a point $p_\epsilon \in conv(S)$ such that
\begin{equation} \label{TA}
\Vert p - p_\epsilon \Vert \leq \epsilon \Vert p - v_i \Vert, \quad \text{for some $i=1, \dots, n$};
\end{equation}
or prove that $p \not \in conv(S)$. The complexity of the Triangle Algorithm in computing an approximate solution as in (\ref{TA}), assuming such solution exists, is $O(mn\epsilon^{-2})$ arithmetic operations.  This complexity is attractive in contrast with the complexity of algorithms that solve the convex hull problem exactly.

The convex hull decision problem is a basic and fundamental problem in computational geometry as well as in linear programming (LP).  More general  convex hull problems are described in Goodman and O'Rourke \cite{Goodman}.  On the one hand the convex hull decision problem is a very special case of LP feasibility.  On the other hand, it is well known that through LP duality theory the general LP problem may be cast as a single LP feasibility problem, see e.g. Chv\'atal \cite {chvatal2}.  This problem in turn is related to the convex hull decision problem. In fact it can be justified that the two most famous polynomial-time LP algorithms, the ellipsoid algorithm of Khachiyan  \cite{kha79} and the projective algorithm of Karmarkar \cite{kar84}, are in fact explicitly or implicitly designed to  solve a case of the convex hull decision problem where $p=0$, see \cite{jinkal}.  In \cite{kha90} another polynomial-time algorithm is given for this special homogeneous case based on diagonal matrix scaling. For integer inputs all known polynomial-time LP algorithms would exhibit theoretical complexity that is polynomial in $m$, $n$, and the size of encoding of the input data, often denoted by $L$, see e.g. \cite{kha79}. The number $L$ can generally be taken to be dependent on the logarithm of $mn$ and of the logarithm of the absolute value of the largest entry in the input data. No {\it strongly polynomial-time} algorithm is known for LP, i.e. an algorithm that would in particular solve the convex hull decision problem in time complexity polynomial in $m$ and $n$.

The convex hull decision problem   can also be formulated as the minimization of a convex quadratic function over a simplex. This particular convex program  has found  applications in statistics, approximation theory, and machine learning,  see e.g Clarkson \cite{clark2008} and Zhang \cite{zhang} who consider the analysis of a greedy algorithm for the more general problem of minimizing certain convex functions over a simplex (equivalently, maximizing  concave functions over a simplex). The oldest version of such greedy algorithms is Frank-Wolfe algorithm, \cite{Frank}.   Special cases of the problem include  support vector machines (SVM), and approximating functions as convex combinations  of other functions, see e.g. Clarkson \cite{clark2008}. The problem of computing the closest point of the convex hull of a set of points to the origin, known as {\it polytope distance} is the case of the convex hull decision problem  where $p$ is the origin. In some applications the  polytope distance refers to the distance  between two convex hulls. Gilbert's algorithm \cite{Gilbert}, \cite{Gil2}  for the polytope distance problem is one of the earliest known algorithms for the problem. G{\"a}rtner and Jaggi \cite{Gartner} show Gilbert's algorithm coincides with Frank-Wolfe algorithm when applied to the minimization of a convex quadratic function over a simplex. In this case the algorithm is known as {\it sparse greedy approximation}.

From the description of Gilbert's algorithm in \cite{Gartner} it does not follow  that Gilbert's algorithm and the Triangle Algorithm are identical.  However, there are similarities in theoretical performance of the two algorithms and  in \cite{kal12}  we give such theoretical comparisons between the Triangle Algorithm for the convex hull decision problem and the sparse greedy approximation.  Indeed we believe that the simplicity of the Triangle Algorithm and a new duality theorem that inspires the algorithm, as well as its theoretical performance all make it distinct from other algorithms for the convex hull decision problem. Furthermore, these features of the Triangle Algorithm may encourage and inspire new applications of the algorithm. The present article is testimonial to this claim where we will use the algorithm to solve a linear system.

In what follows we present two iterative algorithms for solving an invertible matrix equation  $Ax=b$ in dimension $n$ via the Triangle Algorithm.  In the first algorithm we convert $Ax=b$ into a convex hull problem. Then given $\epsilon >0$,  we apply the Triangle Algorithm and a {\it sensitivity theorem} from \cite{kal12} to compute in $O(n^2\epsilon^{-2})$ arithmetic operations an approximate solution satisfying $\Vert Ax_\epsilon - b \Vert \leq \epsilon \rho$, where $\rho= \max \{ \Vert b \Vert, \Vert a_1 \Vert, \dots, \Vert a_n \Vert \}$, and $a_i$ is the $i$-th column of $A$.  In the second algorithm we apply the Triangle Algorithm incrementally, solving a sequence of convex hull problems while using a {\it distance duality} from \cite{kal12}.  The first algorithm requires an a priori estimate of the least coordinate of the solution while the second algorithm builds up this estimate iteratively.  These offer alternatives to exact methods and to iterative methods such as the Jacobi, Gauss-Seidel, and successive over-relaxation methods, requiring none of their structural restrictions. The extreme  simplicity and  theoretical complexity bounds suggest the potential practicality of the methods, especially for large scale or sparse systems. The assessment of computational performance is the subject of future
experimentations. In \cite{Kalan12} we considered an application of the Triangle Algorithm for solving a very special case of a matrix equation where the vector $b$ is unknown but some attributes such as its Euclidean norm is available.

This article is organized as follows. In Section \ref{sec2} we describe the Triangle Algorithm from \cite{kal12}, a duality that it makes use of, the {\it distance duality},  as well as some geometric properties of this duality. In Section \ref{sec3}  we consider solving a linear system, $Ax=b$.  In \ref{subsec1} we consider the case where $x=A^{-1}b$ is known to be nonnegative and show how the Triangle Algorithm together with a {\it sensitivity theorem} from \cite{kal12} solves a linear system and state its complexity for computing an approximate solution. In \ref{subsec2} we consider the general case of $Ax=b$ where there is no information on the solution and show how this can be  converted into the first case by computing an a priori lower bound $t_*$ on the minimum component of the solution vector $x=A^{-1}b$. Then we state a complexity result.  In \ref{subsec3} we offer an incremental version of the Triangle Algorithm for solving $Ax=b$ not requiring an a priori lower bound on $t_*$.  We conclude with some final remarks.

\section{Review of The Triangle Algorithm} \label{sec2}

Throughout we let $\Vert \cdot \Vert$  denote the Euclidean norm.  In \cite{kal12}  we proved the following characterization theorem for the convex hull decision problem.

\begin{thm} \label{thm1} {\bf (Distance Duality \cite{kal12} )} Let $S= \{v_1, \dots, v_n\} \subset \mathbb{R} ^m$, $p \in \mathbb{R} ^m$. Then we have

(i): $p \in conv(S)$ if and only if given any $p' \in conv(S)$, there exists $v_j$ such that $\Vert p'- v_j \Vert  \geq  \Vert p - v_j \Vert $.

(ii): $p \not \in conv(S)$ if and only if there exists $p' \in conv(S)$ such that $\Vert p' -v_i \Vert  < \Vert p - v_i \Vert $, $\forall$ $i$. \qed
\end{thm}

\begin{definition} We say $p' \in conv(S)$ is a {\it witness}  if it satisfies $\Vert p' - v_i \Vert  < \Vert p - v_i \Vert$, for all $i=1, \dots, n$.
\end{definition}

Each witness certifies the infeasibility of $p$ in $conv(S)$. The next theorem shows that each witness actually induces a separating hyperplane. The set $W_p$ of all such witnesses is the intersection of $conv(S)$ and open balls $B_i$ of radius $ \Vert p - v_i \Vert$ centered at $v_i$, $i=1, \dots, n$  and forms a convex open set in the relative interior of $conv(S)$.

\begin{thm}  \label{thmWit} {\bf (Characterization of Witness Set \cite{kal12})}
$p' \in W_p$ if and only if the orthogonal bisector  hyperplane of the line segment $pp'$ separates $p$ from $conv(S)$. \qed
\end{thm}

\begin{figure}[htpb]
	\centering
	
	\begin{tikzpicture}[scale=0.35]
%\draw (0, 10) rectangle (10, 0);	
	
%%here

		\begin{scope}[red]
		\draw (0.0,0.0) circle (4.177319714841085);
		\draw (7.0,0.0) circle (3.008321791298265);
		\draw (4.0,3.0) circle (2.2278820596099706);
		\end{scope}
		
		\draw (0.0,0.0) -- (7.0,0.0) -- (4.0,3.0) -- cycle;
		\draw (0,0) node[below] {$v_1$};
		\draw (7,0) node[below] {$v_2$};
		\draw (4,3) node[above] {$v_3$};
		\filldraw (4.1,0.8) circle (2pt) node[left] {$p$};
\filldraw (0,0) circle (2pt);
\filldraw (7,0) circle (2pt);
\filldraw (4,3) circle (2pt);
%%tohere

\begin{scope}[red]
		 \clip  (19.5,0.0) circle (5.1);
		 \clip (26.5,0.0) circle (7.8);
		 \clip (23.5,3.0) circle (3.6);
          \clip (19.5,0.0) -- (26.5,0.0) -- (23.5,3.0) -- cycle;
\fill[color=gray!29] (-30, 10) rectangle (30, -10);
%\fill[color=gray];
\end{scope}

\begin{scope}[red]
         \draw (19.5,0.0) circle (5.1);
		 \draw (26.5,0.0) circle (7.8);
		 \draw (23.5,3.0) circle (3.6);
\end{scope}
			
\draw (19.5,0.0) -- (26.5,0.0) -- (23.5,3.0) -- cycle;
		\draw (19.5,0) node[below] {$v_1$};
		\draw (26.5,0) node[below] {$v_2$};
		\draw (23.5,3) node[above] {$v_3$};
\filldraw (19.5,0) circle (2pt);
\filldraw (26.5,0) circle (2pt);
\filldraw (23.5,3) circle (2pt);
\filldraw (20.5,5) circle (2pt) node[above] {$p$};
%%here is new
	\end{tikzpicture}
	\caption{Examples of empty $W_p$ ($p \in conv(S)$) and nonempty  $W_p$ ($p \not \in conv(S)$), gray area.}
\label{Fig2}
\end{figure}
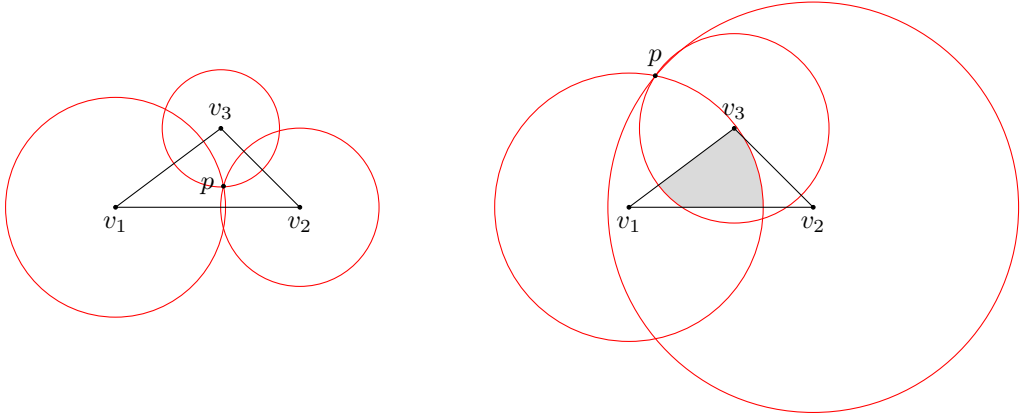

The following is a simple but useful consequence of Theorem \ref{thmWit}:

\begin{cor}  \label{cornew} Suppose $p \not \in conv(S)$.  Let
\begin{equation}
\Delta= \min \{\Vert p - x \Vert: \quad x \in conv(S)\}.
\end{equation}
Then any $p' \in W_p$ gives an estimate of $\Delta$ to within a factor of two. More precisely,
\begin{equation} \label{halfapprox}
\frac{1}{2} d(p,p') \leq  \Delta \leq d(p,p'). \qed
\end{equation}
\end{cor}

\begin{definition} Given $\epsilon \in (0,1)$ we say $p' \in conv(S)$ is an $\epsilon$-approximate solution if it satisfies
\begin{equation}
\Vert p'-p \Vert \leq \epsilon R,
\end{equation}
where
\begin{equation}
R= \max \big \{\Vert p- v_1 \Vert, \dots, \Vert p- v_n \Vert \big \}.
\end{equation}
\end{definition}

Using the characterization theorem (Theorem \ref{thm1}), in \cite{kal12} we described a simple algorithm, called the  {\it Triangle Algorithm}.
Given a desired tolerance $\epsilon \in (0,1)$,  and a current {\it iterate} $p' \in conv(S)$, in each iteration the Triangle Algorithm searches for a triangle $\triangle pp'v_j$ where $v_j \in S$ satisfies $\Vert p' - v_j \Vert \geq \Vert p - v_j \Vert $. Given that such triangle exists, the algorithm uses $v_j$ as a {\it pivot point} to ``pull'' the current iterate $p'$ closer to $p$ to get a new iterate $p'' \in conv(S)$.  If no such a triangle exists, then by Theorem \ref{thm1}, $p'$ is a witness certifying that $p$ is not in $conv(S)$. The Triangle Algorithm consists of iterating two steps:

\begin{center}
\begin{tikzpicture}
%\begin{center}
\node [mybox] (box){%
    \begin{minipage}{0.9\textwidth}
    \begin{center}
    \underline{
{\bf  Triangle Algorithm ($S=\{v_1, \dots, v_n\}$, $p$)}}\
\end{center}
\begin{itemize}
\item
{\bf Step 1.}
Given an {\it iterate} $p' \in conv(S) \setminus \{p\}$, check if there exists a
{\it pivot point} $v_j \in S$, i.e. $\Vert p' -v_j \Vert \geq  \Vert p - v_j \Vert$. If no such $v_j$ exists, then $p'$ is a {\it witness}, stop.

\item
{\bf Step 2.}
Otherwise, compute the {\it step-size}
\begin{equation}
\alpha = \frac{(p-p')^T(v_j-p')}{\Vert v_j - p' \Vert^2}.
\end{equation}
Let the {\it iterate} be defined as
\begin{equation} \label{pdp}
p'' =
\begin{cases}
(1-\alpha)p' + \alpha v_j, &\text{if $\alpha \in [0,1]$;}\\
v_j, &\text{otherwise.}
\end{cases}
\end{equation}
Replace $p'$ with $p''$, go to Step 1.
\end{itemize}

    \end{minipage}};
%\end{center}
\end{tikzpicture}
\end{center}

By an easy calculation that shift $p'$ to the origin, it follows that
the point $p''$ in Step 2 is the closest point to $p$ on the line $p'v_j$, see Figure \ref{Fig5}.

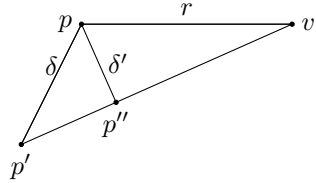
\begin{figure}[htpb]
	\centering
	\begin{tikzpicture}[scale=0.4]
%\draw (0, 10) rectangle (10, 0);	

%\draw (0, 10) rectangle (10, 0);	
	
		\draw (0.0,0.0) -- (7.0,0.0) -- (-2.0,-4.0) -- cycle;
      \draw (0,0) -- (7,0) node[pos=0.5, above] {$r$};
      \draw (-2,-4) -- (0,0) node[pos=0.5, above] {$\delta$};
       \draw (0,0) -- (1.15,-2.6) node[pos=0.5, right] {$\delta'$};
       \draw (1.15,-2.6) node[below] {$p''$};
       \filldraw (1.15,-2.6) circle (2pt);
		\draw (0,0) node[left] {$p$};
		\draw (7,0) node[right] {$v$};
		\draw (-2,-4) node[below] {$p'$};
%         \draw (14,0) node[right]{$C'$};
 %         \draw (-7,0) node[left]{$C$};
 %          \draw (-4.48,0) node[left]{$C''$};
           \filldraw (0,0) circle (2pt);
\filldraw (7,0) circle (2pt);
\filldraw (-2,-4) circle (2pt);
		
		%\filldraw (-2,7) circle (1pt) node[left] {$p$};
	\end{tikzpicture}
\begin{center}
\caption{Depiction of gaps $\delta=d(p',p)$, $\delta'=d(p'',p)$.} \label{Fig5}
\end{center}
\end{figure}

Since $p''$ is a convex combination of $p'$ and $v_j$, it will remain in $conv(S)$.  The algorithm replaces $p'$ with $p''$ and repeats the iterative step.  Note that a pivot point $v_j$ may or may not be a vertex of $conv(S)$. In fact
 $p''$ can be explicitly written as a convex combination of $v_i$'s,
\begin{equation} \label{eqrep}
p''=   \sum_{i=1}^n \beta_i v_i, \quad  \beta_j=(1-\alpha)\alpha_j+\alpha,  \quad \beta_i= (1-\alpha)\alpha_i,  \quad \forall i \not =j.
\end{equation}

The following straightforward result implies that testing for a pivot point or a witness we do not need to compute square-roots.
\begin{prop}  A point $v_j$ is a pivot point if and only if
\begin{equation} \label{eqaa11}
(p-p')^Tv_j \geq \frac{1}{2}(\Vert p \Vert^2- \Vert p'\Vert^2).
\end{equation}
Equivalently, $p'$ is a witness if and only if
\begin{equation} \label{eqaa12a}
(p-p')^Tv_i < \frac{1}{2}(\Vert p \Vert^2- \Vert p'\Vert^2), \quad \forall i=1, \dots, n.
\end{equation}
\end{prop}

\begin{remark}  \label{rmk1} If $p=0$, then $v_j$ is a pivot point if and only if
\begin{equation} \label{eqaa111}
p'^Tv_j \leq  \frac{1}{2}\Vert p'\Vert^2.
\end{equation}
In particular, if $p'^Tv_j <0$, then $v_j$ is a pivot point.
Equivalently, $p'$ is a witness if and only if
\begin{equation} \label{eqaa12a2}
p'^Tv_i > \frac{1}{2}\Vert p'\Vert^2, \quad \forall i=1, \dots, n.
\end{equation}
\end{remark}

The following theorem in \cite{kal12} estimates the complexity of the Triangle Algorithm.

\begin{thm}   \label{thm4} The Triangle Algorithm correctly solves  the convex hull problem  as follows:

(i) Suppose $p \in conv(S)$. Given $\epsilon >0$, the number of iterations $K_\epsilon$ to compute a point $p_\epsilon$ in $conv(S)$ so that

\begin{equation} \label{eqaa3}
\Vert p_\epsilon - p \Vert  \leq \epsilon \Vert p - v_j \Vert,
\end{equation}
for some $v_j \in S$ satisfies
\begin{equation} \label{iter1}
K_\epsilon \leq   \frac{48}{\epsilon^2} =  O(\epsilon^{-2}).
\end{equation}
(ii) Suppose $p \not \in conv(S)$. If $\Delta= \min \big \{\Vert x - p \Vert : \quad x \in conv(S) \big \}$,
the number of iterations $K_\Delta$ to compute a witness, a point
$p_\Delta$ in $conv(S)$ so that $d(p_\Delta,v_i) < d(p, v_i)$ for all $v_i \in S$, satisfies
\begin{equation} \label{iter2}
K_\Delta  \leq  \bigg  \lceil \frac{8R^2}{\Delta^2} \ln \bigg (\frac{2 \delta_0}{\Delta} \bigg ) \bigg \rceil. \qed
\end{equation}
\end{thm}

\section{Solving A Linear System  Via The Triangle Algorithm} \label{sec3}

Consider solving $Ax=b$ with $A$ an invertible $n \times n$ real matrix. Let  $A=[a_1, a_2, \dots, a_n]$, where $a_i \in \mathbb{R} ^n$.
\begin{definition} \label{defaprx} We say $x_0$ is an  $\epsilon$-approximate solution of $Ax=b$ if
\begin{equation} \label{eddef1}
\Vert A x_0 - b \Vert \leq \epsilon \rho, \quad \rho=\max \big \{\Vert a_1 \Vert, \dots, \Vert a_n \Vert, \Vert b \Vert \big \}.
\end{equation}
\end{definition}

\subsection{Solving A Linear System With Nonnegative Solution} \label{subsec1}

First, we assume that $x=A^{-1}b \geq 0$ and show how to solve this as a convex hull problem. Next, we solve the general case relaxing this condition.  Since $A$ is invertible, $Ax=0$ has only the trivial solution. In particular,
\begin{equation} \label{recession}
0 \not \in conv \big (\{a_1, \dots, a_n\} \big ).
\end{equation}

In \cite{kal12} we described an application of the Triangle Algorithm in computing an approximate feasible solution of the feasibility problem in linear programming. Here we use that approach to solve $Ax=b$ via the Triangle Algorithm to give for any $\epsilon \in (0,1)$, an $\epsilon$-approximate solution. Clearly, by adjusting $\epsilon$  we can compute an approximate solution with absolute error of $\epsilon$, however the approximate solution within a relative error as defined in Definition \ref{defaprx} is a more sensible measure of approximation.  The following auxiliary result is easy to prove.

\begin{prop}  Assume $x=A^{-1}b \geq 0$.  Then $0 \in conv \big (\{a_1, \dots, a_n, -b \} \big )$. $\qed$
\end{prop}
It follows that solving $Ax=b$ approximately is equivalent to finding an appropriate approximation to $0$ in the set
$conv \big (\{a_1, \dots, a_n, -b\})$.

The following sensitivity theorem establishes the needed accuracy to which an approximate solution in $conv \big (\{a_1, \dots, a_n, -b \} \big )$ should be computed. It is tailored for the case of solving $Ax=b$, $x=A^{-1}b \geq 0$.

\begin{thm} \label{thmc} {\bf (Sensitivity Theorem \cite{kal12})} Let
\begin{equation} \label{deltaz}
\Delta_0 = \min \big \{\Vert p \Vert:  \quad p \in conv \big (\{a_1, \dots, a_n\} \big)  \big \},
\end{equation}
\begin{equation} \label{rprime}
\rho= \max \big \{\Vert a_1 \Vert, \dots, \Vert a_n \Vert,  \Vert b \Vert  \big \}.
\end{equation}
Let $\Delta_0'$ be any number such that $0 < \Delta_0' \leq \Delta_0$.
Suppose $\epsilon \in (0,1)$ satisfies
\begin{equation}  \label{rprimebd}
\epsilon  \leq \frac{\Delta_0'}{2 \rho}.
\end{equation}
Suppose we have computed
\begin{equation} \label{eqthmc1}
p'= (\alpha_1 a_1+ \cdots + \alpha_n a_n) - \alpha_{n+1} b  \in conv (\{a_1, \dots, a_n, -b\})
\end{equation}
satisfying
\begin{equation} \label{eqthmc2}
\Vert p' \Vert \leq \epsilon \rho
\end{equation}
Let
\begin{equation} \label{eqthmc3}
x_0=(\frac{\alpha_1}{\alpha_{n+1}}, \dots, \frac{\alpha_n}{\alpha_{n+1}})^T.
\end{equation}
Then, $x_0 \geq 0$, and if
\begin{equation} \label{eqthmc3A}
\epsilon'= 2\bigg (1+ \frac{\Vert b \Vert }{\Delta'_0} \bigg ) \epsilon,
\end{equation}
we have
\begin{equation} \label{eqthmc4}
\Vert Ax_0 - b \Vert \leq \epsilon' \rho,
\end{equation}
i.e. $x_0$ is an $\epsilon'$-approximate solution of $Ax=b$. \qed
\end{thm}

Now consider the following Two-Phase Triangle Algorithm:

\begin{center}
\begin{tikzpicture}
%\begin{center}
\node [mybox] (box){%
    \begin{minipage}{0.9\textwidth}
\begin{center}
\underline{
{\bf  Two-Phase Triangle Algorithm ($A=[a_1, \dots, a_n]$, $b$)}}
\end{center}

\begin{itemize}
\item
{\bf Phase 1.} Call {\bf Triangle Algorithm}($\{a_1, \dots, a_n\}, 0$) to get a witness $p'$ $\in conv (\{a_1, \dots, a_n\})$.

\item
{\bf Phase 2.} Starting with $p'$ in Phase 1, call {\bf Triangle Algorithm}($\{a_1, \dots, a_n, -b\}, 0$).
\end{itemize}

    \end{minipage}};
%\end{center}
\end{tikzpicture}
\end{center}

The first phase of the algorithm attempts to find a witness $p' \in conv(\{a_1, \dots, a_n\})$ that proves $0$ is not in this convex hull. Any such witness $p'$ according to (\ref{halfapprox}) gives rise to a lower bound to $\Delta_0$ which in turn can be used to select $\epsilon$, see (\ref{rprimebd}) in Theorem \ref{thmc}. From this and the complexity result in Theorem \ref{thm4}, we have:

\begin{thm} \label{thm2phase} Given any $\epsilon_0 \in (0,1)$, in order to compute an $\epsilon_0$-approximate solution (i.e. a solution $x_0 \geq 0$ such that $\Vert Ax_0-b \Vert \leq \epsilon_0 \rho $)
it suffices to set  $\Delta'_0=0.5d(p',0)$, where $p'$ is the witness computed in Phase 1 of the Two-Phase Triangle Algorithm.  Then in Phase 2 of the algorithm it suffices to compute a point $p' \in conv(\{a_1, \dots, a_n, -b\}$ so that
\begin{equation}
\Vert p' \Vert \leq \epsilon \rho ,
\end{equation}
where $\epsilon$ satisfies
\begin{equation}
\epsilon \leq \frac{\Delta_0'}{2}\min \bigg \{ \frac{1}{\rho}, \quad \frac{\epsilon_0}{(\Delta_0' + \Vert b \Vert )} \bigg \}.
\end{equation}

In particular, since $\Delta_0 \leq \rho$, it suffices to pick
\begin{equation}
\epsilon  \leq \frac{\Delta_0'}{4 \rho} \epsilon_0.
\end{equation}
Then the number of iterations in Phase 2 of the algorithm, $K_\epsilon$, each of cost $O(n^2)$ arithmetic operations, satisfies
\begin{equation}
K_\epsilon \leq  \frac{48}{{\epsilon_0}^2} \frac{\rho^2}{{\Delta'_0}^2}. \qed
\end{equation}
\end{thm}

\begin{remark}  The complexity of Phase 1 is dominated by that of Phase 2. To see this note that from Theorem \ref{thm4} to compute a witness in testing if $0$ lies in $conv(\{a_1, \dots, a_n\})$, the algorithm requires $K_{\Delta_0}$ iterations.  Since $\epsilon \leq \Delta'_0/ 2 \rho$, and $\delta_0 \leq \rho$, it follows that
\begin{equation}
\ln \frac{2\delta_0} {\Delta'_0} \leq \ln \frac{1}{\epsilon}.
\end{equation}
Thus $K_{\Delta_0}$ is bounded above by the bound on $K_\epsilon$ in Theorem
\ref{thm2phase}.
\end{remark}

\begin{remark}   In practice, with the goal of computing an $\epsilon$-approximate solution, we do not need to compute an a priori estimate $\Delta_0'$ on $\Delta_0$  to be used in Theorem \ref{thm2phase}.  We could forgo Phase 1 altogether and simply run Phase 2. Assuming that $A$ is invertible, in each iteration we compute from the current iterate $p'$ the current approximate solution $x_0$,  then check if it is an $\epsilon_0$-approximate solution of $Ax=b$.
\end{remark}

The following result, while not practical, suggests a lower bound to $\Delta_0$ can be stated based on the smallest eigenvalue of $A^TA$.

\begin{prop}  Let $\Delta_0$ be as in (\ref{deltaz}).  Let $\lambda_{\min}$ be the minimum eigenvalue of the matrix $Q=A^TA$. Then
\begin{equation} \label{delta1}
\Delta_0= \min \big \{\Vert Ax\Vert:  \quad \sum_{i=1}^n x_i =1, \quad x_i \geq 0  \big \}  \geq \frac{1}{\sqrt{n}} \lambda_{\min}.
\end{equation}
\end{prop}
\begin{proof}
Suppose $\Delta_0=\Vert Ax'\Vert$, where $\sum_{i=1}^n x'_i=1$, $x'\geq 0$. It is easy to prove $\Vert x' \Vert \geq 1/\sqrt{n}$. We thus have
\begin{equation}
\Vert Ax'\Vert^2=x'^TQx'= \Vert x' \Vert^2 (\frac{x'}{\Vert x' \Vert})^T Q (\frac{x'}{\Vert x' \Vert}) \geq \frac{1}{n} \lambda^2_{\min}.
\end{equation}
\end{proof}

Before considering the general case of solving $Ax=b$ we consider a small example.

\begin{example} Consider the $2 \times 2$ linear system.
\begin{equation}
 \begin{pmatrix}
  ~3 & -2  \\
  ~2 & ~1
 \end{pmatrix}
  \begin{pmatrix}
  x_1 \\
  x_2
 \end{pmatrix}
 =
 \begin{pmatrix}
  -1 \\
  ~4
 \end{pmatrix}
\end{equation}
Its solution is $x=(1,2)$.
Thus the Two-Phase Triangle Algorithm  can compute an approximate solution to any prescribed accuracy. We will consider one iteration, skipping Phase 1. To compute an approximate solution we test via the Triangle Algorithm if $0$ lies in the convex hull of the set
\begin{equation}
\bigg \{
 a_1= \begin{pmatrix} 3 \\
  2
 \end{pmatrix},  a_2=\begin{pmatrix}
  -2 \\
  ~1
 \end{pmatrix},  -b=\begin{pmatrix}
  ~1 \\
  -4
 \end{pmatrix}  \bigg \}.
 \end{equation}
 Let the initial iterate $p'$ be selected as the center of mass, $p'=(a_1+a_2-b)/3 =(2/3, -1/3)^T$.  Thus the initial approximate solution to $Ax=b$ is $x_0=(\alpha_1/\alpha_2, \alpha_2/\alpha_3)^T=(1,1)^T$. The corresponding error is $\Vert Ax_0-b \Vert= \sqrt{5}$.  In Step 1 of the Triangle Algorithm we select $a_2$ as the pivot point since from Remark \ref{rmk1} we have
 \begin{equation}
 p'^Ta_2 <0.
 \end{equation}
 The corresponding step size is
 \begin{equation}
 \alpha =-\frac{p'^T(a_2-p')}{\Vert a_2-p' \Vert^2}= \frac{1}{4}.
 \end{equation}
 Thus
 \begin{equation}
 p''=(1-\frac{1}{4})p'+ \frac{1}{4}a_2 = \frac{1}{4}a_1 + \frac{1}{2}a_2 - \frac{1}{4}b.
 \end{equation}
 This implies the next approximate solution to $Ax=b$ is $x_1=(1,2)$ (using that the new coefficients are $\alpha_1=1/4$, $\alpha_2=1/2$, $\alpha_3= 1/4$) which is indeed the exact solution.
\end{example}

\subsection{Solving A General Linear System Via The Triangle Algorithm} \label{subsec2}

In this section we consider the general case of solving $Ax=b$ with $A$ an invertible matrix, where it is not known if the solution $x=A^{-1}b$ is nonnegative.  Let $e=(1, \dots, 1 )^T \in \mathbb{R}^m$.  Then there exists
$t \geq 0$ such that if $x$ is the solution to
\begin{equation} \label{eqx1}
A(x-te)=b,
\end{equation}
then $x \geq 0$.  Thus if we let $u=Ae$, then
\begin{equation} \label{eqx2}
Ax=b+tu, \quad x \geq 0,
\end{equation}
is solvable.  Since $A$ is invertible $u \not =0$. Let
\begin{equation} \label{eqtstar}
t_*= \min \big \{t: A(x-te)=b, \quad x \geq 0 \big \}.
\end{equation}

If a value $t$ is known such that $t \geq t_*$,  then we can apply the Two-Phase Triangle Algorithm  to solve  (\ref{eqx2}).  In the complexity analysis we use bounds that also depend upon $t$. Specifically, set
\begin{equation}
b(t)=b+tu.
\end{equation}
Then we may restate Theorem \ref{thmc} as well as the complexity result, Theorem \ref{thm2phase},  with $\Vert b \Vert$ and $\rho$ replaced with $\Vert b(t) \Vert $ and $\rho(t)$ defined as
\begin{equation} \label{rprime3}
\rho(t)= \max \big \{\Vert a_1 \Vert, \dots, \Vert a_n \Vert,  \Vert b(t) \Vert  \big \}.
\end{equation}

For any given $t \geq 0 $ we can state  the following complexity in testing the solvability of (\ref{eqx2}).

\begin{thm} \label{thm2primet}
Given $\epsilon_0 >0$, and any $t \geq 0$, and any lower bound $\Delta_0'$, $0 <\Delta_0' \leq  \Delta_0$, the  Two-Phase Triangle Algorithm in at most
$$\frac{48}{\epsilon_0^2} \frac{\rho(t)^2}{{\Delta'_0}^2}$$
iterations, each of cost $O(n^2)$ arithmetic operations, either determines that $Ax =b+tu$, $x \geq 0$ is infeasible, or computes $x_0$ satisfying
$$\Vert Ax_0 - b \Vert \leq \epsilon_0 \rho.$$
\end{thm}

Next we show how to actually find an upper  bound on $t_*$ defined in (\ref{eqtstar}).

\begin{thm}  \label{thm7} Consider  $Ax=b$ and assume $A$ is invertible. Let $x^*=A^{-1}b$ and $x^*_i$ be its $i$-th component.
Let $Q=A^TA$, $w=A^Tb$, and let the $i$-th column of $Q$ be denoted by $q_i$.  Let $q_{\min}= \min \big \{\Vert q_j \Vert, j=1, \dots, n \big \}$.
Let $\lambda_{\min}$ be the minimum eigenvalue of $Q$. Set

\begin{equation} \label{eq11}
\tau_*=\bigg ( \prod_{j=1}^n \Vert q_j \Vert \bigg ) \frac{\Vert w \Vert}{q_{\min} \lambda^n_{\min}}, \quad  \tau'_*=\bigg ( \prod_{j=1}^n \Vert q_j \Vert \bigg ) \frac{\Vert w \Vert}{q_{\min} \det (Q)}.
\end{equation}
Then
\begin{equation} \label{eqxtar}
x^*_i \geq - \tau'_* \geq -\tau_*.
\end{equation}
In particular, $\tau_*  \geq  \tau'_*  \geq t_*$.
\end{thm}
\begin{proof}   We have $Qx^*=w$. By the Cramer's rule we have

\begin{equation}
x_i^*= \frac{\det(Q_i)}{\det(Q)},
\end{equation}
where $Q_i$ is the matrix that replaces column $i$ of $Q$ with the vector $w$. Thus, \begin{equation}
x_i^* \geq - \frac{|\det(Q_i)|}{\det(Q)}.
\end{equation}
By Haddamard's inequality on determinants we have
\begin{equation}
|\det(Q_i)| \leq \bigg (\prod_{j=1}^n \Vert q_j \Vert  \bigg ) \frac{\Vert w \Vert}{\Vert q_i \Vert} \leq \bigg (\prod_{j=1}^n \Vert q_j \Vert \bigg ) \frac{\Vert w \Vert}{q_{\min}}  .
\end{equation}
On the other hand, $\det(Q) = \prod_{j=1}^n \lambda_j$, where $\lambda_1, \dots, \lambda_n$ is the set of eigenvalues of $Q$.  Since $Q$ is positive definite, these are positive and we have  $\det(Q) \geq  \lambda^n_{min}$. These imply
(\ref{eqxtar}).  Since $A((x^*+\tau'_*e)-\tau'_*e)=b$, from (\ref{eqtstar}) it follows that  $x^*+\tau'_*e  \geq 0$. Hence, $\tau'_* \geq t_*$. But also $\tau_* \geq \tau'_*$.
\end{proof}

\subsection{An Incremental Triangle Algorithm For Solving a Linear System} \label{subsec3}

In this section we consider solving $Ax=b$ without using an a priori upper bound $t$ on the value of $t_*$ needed to make  the solution of $Ax=b+tu$ nonnegative (see (\ref{eqx2})).  That is, rather than selecting a specific value for $t \geq t_*$, we gradually increase its value (starting from $t=0$) and solve the corresponding system $Ax=b+tu$ via the Triangle Algorithm to get an $\epsilon$-approximate solution. If for a particular value of $t$  the Triangle Algorithm fails to generate an $\epsilon$-approximate solution we then increase $t$ and repeat the process. However, in doing so we make use of the duality implied by Theorem \ref{thm1}, i.e. the existence of a witness computed via the Triangle Algorithm.

More specifically, the {\it Incremental Triangle Algorithm} works as follows.  Assume that for a given $t_0 \geq 0$ (initially set be to be zero)  we have attempted to compute an $\epsilon$-approximate solution for $Ax=b$, i.e. a vector $x_0 \geq 0$ such that
$$\Vert A(x_0-t_0e) - b \Vert \leq \epsilon  \max \big \{\Vert a_1 \Vert, \dots, \Vert a_n \Vert, \Vert b \Vert \big \}.$$
If this is possible we are done. If not, then $conv  \big (\{a_1, \dots, a_n, -b(t_0) \big \})$  does not contain the origin, where $b(t_0)=b+t_0u$.  Thus by Theorem \ref{thm1} the Triangle Algorithm computes a witness, i.e.
\begin{equation} \label{eqinc2}
p'(t_0) \in conv \big (\{a_1, \dots, a_n, -b(t_0)\} \big)
\end{equation}
such that the following set of $n+1$ strict inequalities are satisfied:
\begin{equation} \label{eqinc3}
\Vert p'(t_0)- a_i \Vert <  \Vert a_i \Vert, \quad \forall i=1, \dots, n,
\end{equation}
and
\begin{equation} \label{eqinc4}
\Vert p'(t_0)+ b(t_0) \Vert <  \Vert b(t_0) \Vert.
\end{equation}
Equivalently, after expanding and simplifying (\ref{eqinc3}) and (\ref{eqinc4}) we get
\begin{equation} \label{eq3b'}
\Vert p'(t_0) \Vert^2 - 2p'(t_0)^T a_i <0, \quad \forall i=1, \dots, n,
\end{equation}
\begin{equation} \label{eq3c'}
\Vert p'(t_0) \Vert^2 + 2p'(t_0)^T b(t_0) <0.
\end{equation}
From (\ref{eqinc2}) we have
\begin{equation}
p'(t_0)= \sum_{i=1}^n \alpha_i a_i -  \alpha_{n+1} (b+ t_0 u), \quad \sum_{i=1}^{n+1} \alpha_i=1, \quad \alpha_i \geq 0, \quad \forall i.
\end{equation}
Letting
\begin{equation}
p'= \sum_{i=1}^n \alpha_i a_i - \alpha_{n+1}b,
\end{equation}
we may write
\begin{equation}
p'(t_0)=p'- t_0\alpha_{n+1}u.
\end{equation}
Thus,
\begin{equation}
p'(t_0)^Ta_i=p'^Ta_i- t_0\alpha_{n+1}u^Ta_i.
\end{equation}
For each $t$ define
\begin{equation}
p'(t)=p'- t\alpha_{n+1}u.
\end{equation}
For $i=1, \dots, n$, define
\begin{equation} \label{eq3b''}
g_i(t)= \Vert p'(t) \Vert^2 - 2p'(t)^T a_i.
\end{equation}
Also, define
\begin{equation} \label{eq3c''}
g_{n+1}(t)= \Vert p'(t) \Vert^2 + 2p'(t)^T b(t).
\end{equation}
It is easy to verify that for $i=1, \dots, n$ we have
\begin{equation} \label{eq3c3}
g_i(t)=t^2 \alpha^2_{n+1} \Vert u \Vert^2 - 2 t \alpha_{n+1}(p'-a_i)^Tu+ \Vert p' \Vert ^2 - 2 p'^T a_i.
\end{equation}
The coefficient of $t^2$ in $g_{n+1}(t)$ can be shown to be
\begin{equation} \label{eq3c4}
\alpha_{n+1}(\alpha_{n+1}-2)\Vert u \Vert^2 .
\end{equation}
We now prove
\begin{prop} Suppose $\alpha_{n+1} \not =0$. For each $i=1, \dots, n$, we can select a real number $t_i >t_0$ such that $g_i(t_i) \geq 0$.
Moreover, $g_{n+1}(t) \leq 0$ for all $t$.
\end{prop}
\begin{proof}  Consider a quadratic polynomial $q(t)=c_2 t^2+  c_1t + c_0$, where $c_2 >0$.  We claim if $q(t)$ takes a negative value at some point $t_0$, then it must have a real root $t_0' >t_0$.  The proof of this follows from the fact that as $t$ approaches infinity, $q(t)$ approaches infinity so that $q(t)$ changes sign from negative at $t_0$ to a positive value. Thus $q(t)$ must be zero at some $t_0' >t_0$.  For each $i=1, \dots, n$, the coefficient of the quadratic term  in $g_i(t)$ is positive. This follows from the fact that $u \not =0$ and $\alpha_{n+1} \not =0$, see (\ref{eq3c3}). Hence applying the argument on $q(t)$ we conclude that each $g_i(t)$, $i=1, \dots, n$, has a root $t_i$ larger than $t_0$.

Since the coefficient of the quadratic term in  $g_{n+1}(t)$  is $\alpha_{n+1} (\alpha_{n+1}-2) \Vert u \Vert^2$ and $ 0 < \alpha_{n+1} < 1$, the coefficient is negative. Since $g_{n+1}(t_0) <0$, it follows that $g_{n+1}(t) <0$ for all $t$.
\end{proof}

Let $t_0'=\min\{t_1, \dots, t_n\}$.  Increasing $t_0$ to $t_0'$, at least for one $i = 1, \dots, n$ we will have $g_i(t'_0) \geq 0$.  Once $t_0'$ is computed we have a pivot point and we reapply the Triangle Algorithm, testing if $Ax=b+t_0'u$, $x \geq 0$ is feasible.  In doing so we can start the Triangle Algorithm with $p'(t_0')=p'+t_0'b$. In practice we do not need to find the exact value of $t_i$ which requires solving the quadratic equations $g_i(t)=0$, $i=1, \dots, n$. We can simply compute an upper bound on the roots. Such upper bounds can be computed easily.  Additionally, we can choose $t_0'$ to be such that $t_0'-t_0$ differs by a natural number, guaranteeing that this iterative process would  eventually select a value $t_0'$ exceeding the value $t_*$, see (\ref{eqtstar}).  A formal description of the Incremental Triangle Algorithm is as follows.\\

\begin{center}
\begin{tikzpicture}
%\begin{center}
\node [mybox] (box){%
    \begin{minipage}{0.9\textwidth}
\begin{center}
\underline{{\bf Incremental Triangle Algorithm($A$, $b$)}}
 \end{center}
\begin{itemize}
\item
{\bf Step 0.} ({\bf Initialization}) Let $u=Ae$, $e=(1, \dots, 1)^T \in \mathbb{R}^m$. Select  $p'=A\alpha - \alpha_{n+1}b  \in conv(S)$, where $\alpha=(\alpha_1, \dots, \alpha_n)^T \in \mathbb{R}^n$, $\sum_{i=1}^{n+1}\alpha_i=1$, $\alpha_i \geq 0$. Set $t_0=0$.

\item
{\bf Step 1.}
Given $p' =A\alpha - \alpha_{n+1}b \in conv(S)$, set $x_0=\alpha/\alpha_{n+1}$. Define $\tau_0$ according to
\begin{equation}
E(\tau_0)=\Vert Ax_0 - (b + \tau_0 u)\Vert  = \min \bigg \{ \Vert Ax_0-(b+tu) \Vert :  \quad t \geq t_0  \bigg \}.
\end{equation}
Replace $t_0$ with $\tau_0$.
If $E(t_0) \leq \epsilon \rho$, set $x_0' =(x_0- t_0 e)$, stop.
\item
{\bf Step 2.} If $p'(t_0)= p'- \alpha_{n+1}t_0u$ is a witness with respect to $conv \big ( \{a_1, \dots, a_n, -(b + t_0 u) \}  \big)$,  go to Step 3. Otherwise,  call Triangle Algorithm ($\{a_1, \dots, a_n, -(b + t_0 u)\}$, $0$)
to compute a new iterate $p''(t_0)$ (see (\ref{pdp}) in Triangle Algorithm):
\begin{equation}
p''(t_0)= p''- \beta_{n+1}t_0u, \quad  \text{where}~p''=A \beta - \beta_{n+1}b \in  conv\big (\{a_1, \dots, a_n, -(b+ t_0 u) \} ) \big).
\end{equation}
Replace $p'$ with $p''$, $\alpha$ with $\beta$, and $\alpha_{n+1}$ with $\beta_{n+1}$. Go to Step 1.
\item
{\bf Step 3.}  Compute $t_0'$, the smallest value $t$ such that $g_i(t) \geq 0$ for some $i=1, \dots, n$ (see (\ref{eq3c3})). Replace $t_0$ with $t_0'$. Go to Step 2.
\end{itemize}

    \end{minipage}};
%\end{center}
\end{tikzpicture}
\end{center}

\begin{remark}
The computation of $\tau_0$ in Step 1 is an auxiliary step to improve the error for a given $x_0$. For a given $\epsilon$, eventually $x_0'$ will give the desired $\epsilon$-approximate solution.
\end{remark}

We will consider an example.  In the example below we will implement the algorithm with and without the auxiliary Step 1 to show different scenarios.

\begin{example} Consider the $2 \times 2$ linear system (\ref{example2}), having $x=(-1,-2)^T$ as its solution.  Thus $t_*=2$.
We consider one iteration in solving $Ax=b$ by the Incremental Triangle Algorithm.
\begin{equation} \label{example2}
 \begin{pmatrix}
  ~2 & -1  \\
  ~1 & ~1
 \end{pmatrix}
  \begin{pmatrix}
  x_1 \\
  x_2
 \end{pmatrix}
 =
 \begin{pmatrix}
  0 \\
  -3
 \end{pmatrix}
\end{equation}
We have $u=Ae=(1, 2)^T$. We set $t_0=0$ and  test if $0$ is in the convex hull of the set
\begin{equation}
\bigg \{
 a_1= \begin{pmatrix} 2 \\
  1
 \end{pmatrix},  a_2=\begin{pmatrix}
  -1 \\
  ~1
 \end{pmatrix},  -b=\begin{pmatrix}
  ~0 \\
  3
 \end{pmatrix}  \bigg \}.
\end{equation}
Suppose we select $p'=(0,3/2)^T$.  It is easy to check that

\begin{equation}
p'=\alpha_1 a_1+ \alpha_2 a_2 - \alpha_3 b, \quad \alpha_1=1/4, \quad \alpha_2= 1/2, \quad \alpha_3=1/4.
\end{equation}
Thus the initial approximate solution to $Ax=b$ is
\begin{equation}
x_0=(\frac{\alpha_1}{\alpha_3}, \frac{\alpha_2}{\alpha_3})^T= (1,2)^T.
\end{equation}
To compute $\tau_0$ we minimize
\begin{equation}
E(t)=\Vert Ax_0-(b + tu) \Vert= \Vert (-t, 6-2t)^T \Vert
\end{equation}
for $t \geq t_0$.  Note that $E(t_0)= 6$. However, the minimum of $E(t)$ is attained at $\tau_0=12/5=2.4$ and $E(\tau_0)=6/\sqrt{5} \approx 2.68$. For simplicity of the computation we round the value of $\tau_0$ to $2$. Next, we replace $t_0$ with $\tau_0$. We have, $E(2)=2\sqrt{2}$.   In Step 2 of the algorithm we have $p'(t_0)=p'-\alpha_3t_0u=(-1/2,1/2)^T$.   From Remark \ref{rmk1}
we may select $a_1$ as the pivot point as we have,
\begin{equation}
p'(t_0)^Ta_1=-1/2 <0.
\end{equation}
The corresponding step size is
\begin{equation}
 \alpha =-\frac{p'(t_0)^T(a_1-p'(t_0))}{\Vert a_1-p'(t_0) \Vert^2}= \frac{2}{13}.
 \end{equation}
 Thus
 \
 $$p''(t_0)=(1-\frac{2}{13})p'+ \frac{2}{13}a_1 = \frac{11}{13}(\frac{1}{4}a_1 + \frac{1}{2}a_2 - \frac{1}{4}(b+ 2u))+ \frac{2}{13}a_1=$$
 \begin{equation}
 \frac{19}{52}a_1 + \frac{11}{26}a_2 - \frac{11}{52}(b+ 2u).
 \end{equation}
 This replaces $p'(t_0)$, and the new $p'$ becomes
\begin{equation}
p'=\frac{19}{52}a_1 + \frac{11}{26}a_2 - \frac{11}{52}b.
\end{equation}
 The next approximation to the linear system is
 $x_0=(19/11, 2)^T$.  The corresponding error, $E(t_0)=\sqrt{936}/11 \approx 2.78$ which is less than $2\sqrt{2}$.  Now we optimize again by computing $\tau_0$ as the minimum of $E(t)= \Vert A x_0-(b+tu) \Vert$. The minimum occurs for $\tau_0=164/55$ and is approximately $1.7$. We replace $t_0$ with this and repeat Step 2.

Rather than continuing with this iteration, we will next consider the same example but without computing $\tau_0$. This will allow us to implement one iteration of Step 3 in this example.  Starting again with the point $p'=(0,3/2)$ we can easily check that it is a witness, i.e.
\begin{equation}
p'^T a_1 > \frac{ \Vert p' \Vert^2}{2}, \quad p'^T a_2 >  \frac{ \Vert p' \Vert^2}{2},
 \quad -p'^T b > \frac{ \Vert p' \Vert^2}{2}.
 \end{equation}
Then
\begin{equation}p'(t)=p'- t\alpha_{3}u=
 (-\frac{t}{4}, \frac{3}{2}-\frac{t}{2})^T.
 \end{equation}
Using this to compute $p'(t)^Ta_i$, for $i=1,2$ we get,
\begin{equation}
g_1(t)=\Vert p'(t) \Vert ^2- 2p'(t)^Ta_1 = \frac{5}{16}t^2+ \frac{1}{2}t - \frac{3}{4},
\end{equation}
\begin{equation}
g_2(t)=\Vert p'(t) \Vert ^2- 2p'(t)^Ta_2 = \frac{5}{16}t^2- t - \frac{3}{4}.
\end{equation}
Also, we  have $b(t)=b+tu=(0,-3)^T+t(1,2)^T=(t,-3+2t)^T$. So we have
\begin{equation}
g_3(t)= \Vert p'(t) \Vert ^2+ 2p'(t)^Tb(t)= -\frac{35}{16}t^2+\frac{15}{2} t - \frac{27}{4}.
\end{equation}
As analyzed for the general case, the quadratic term of $g_3(t)$ has a negative coefficient so that $g_3(t)$ remains negative for all $t$.
We select $t'_0$  to be the positive root of $g_1(t)$, namely $(-8+\sqrt{304})/10 \approx .94$.  The positive root of $g_2(t)$ is larger than this value.

The algorithm then replaces $t_0$ with $t'_0$, and moves to Step 2, setting $p'(t_0)=p'+t_0 \alpha_3u$, then  checking if it is a witness with respect to $ conv({a_1,a_2, -b(t_0)})$. Note that since $t_*=2$,  we expect that if $p'(t_0)$ is not a witness, subsequent iterations of the algorithm (without the optimization step that computes $\tau_0$) must eventually produce a witness.
\end{example}

\begin{remark}  The Incremental Triangle Algorithm avoids using an a priori estimate for $t_*$.  It increases $t$ incrementally, starting with $t=0$, generating a sequence of approximate solutions $x_k$ that converge to the solution of $Ax=b$.  Here $x_k$ and $x_{k+1}$ may correspond to the same value of $t$, or to two consecutive values of $t$.  The increase in the $t$ value in the incremental algorithm is done in a conservative fashion. In order to guarantee that it eventually gets to be large enough, e.g. at least as large as $t_*$ (see (\ref{eqtstar})), or large enough such that an $\epsilon$-approximate solution is possible, we will make sure that the difference between two consecutive $t$ values is at least a natural number $n_0$.

Another possibility in increasing the $t$ value after each report of the infeasibility of $Ax =b(t)$, $x \geq 0$ via a witness, is to double it and add one.
Clearly, this would require at most $O(\log_2 (t_*+1))$ calls to the Triangle Algorithm.  However, the potential advantage in the Incremental Triangle Algorithm is that once for a particular value $t$ we get a witness $p'$, for the next value, say $t'$, if still not sufficiently large, the algorithm will find a new witness $p''$ quickly, in a few iterations.
\end{remark}

\section{Final Remarks}

In this paper we have proposed two novel iterative methods for solving linear system of equations approximately to within prescribed errors. These approximation algorithms are based on the  Triangle Algorithm for a convex hull problem, \cite{kal12}.  Undoubtedly testing the practical performance of these algorithms will require extensive experimentation, as well as comparisons with existing computational results via exact or iterative algorithms.  Ideally, these experimentations should be applied to different types of matrices, from large to sparse matrices.  We hope to carry out some experimentation to assess the performance of the Triangle Algorithm for the convex hull problem, as well as for solving a linear system of equations via the two proposed algorithms. Optimistically, the Triangle Algorithm will perform well in computing approximate solutions to linear systems, perhaps with much better practical performance than its theoretical worst-case complexity.  Potentially, the proposed algorithms can also be combined with the existing iterative algorithms for linear systems.  The simplicity of the Triangle Algorithm and its theoretical performance raises the optimism that it will encourage new research and applications.

\end{document}